\documentclass[reqno]{amsart}
\newcommand \datum {January 15, 2022}
\usepackage{amssymb,latexsym}
\usepackage{amsmath}
\usepackage{graphicx}
\usepackage{color}
\usepackage[dvipsnames]{xcolor} 
\usepackage{enumerate}

\theoremstyle{plain}
 \newtheorem{theorem}{Theorem}
 \newtheorem{lemma}{Lemma}
 
 \newtheorem{corollary}{Corollary}
 
\theoremstyle{definition}

\theoremstyle{remark}

\newcommand \RCSub[1] {\textup{RCSub}(#1)}

\newcommand \prcg[2] {\textup{rcg}_{#1}(#2)}

\newcommand \rcp[1]  {\textup{rc}_{#1}}

\newcommand \ideal [1] {\mathord{\downarrow}#1}
\newcommand \filter [1] {\mathord{\uparrow}#1}
\newcommand \lideal [2] {\mathord{\downarrow}_#1\kern 2pt #2}
\newcommand \lfilter [2] {\mathord{\uparrow}_#1\kern 2pt #2}
\newcommand \ucov[1] {#1^\bullet}
\newcommand \lcov[1] {#1_{\kern 0.5pt\pmb{\pmb\circ}}}
\newcommand \len [1]{\textup{len}(#1)}

\newcommand \llen[1] {\textup{len}_{\kern-0.1pt#1}}
\newcommand \At[1] {\textup{At}(#1)}
\newcommand \Coat[1] {\textup{Coat}(#1)}

\newcommand \Nplu {\mathbb N^+}

\newcommand \restrict [2] {#1\rceil_{\kern -1pt #2}}

\newcommand \set [1]{\{#1\}}

\newcommand \tbf[1]  {\textbf{#1}}

\long\def \nothing #1 {}

\newcommand\red[1]{{\textcolor{red}{#1}}}
\newcommand \mgreen [1] {{\color{green!30!magenta}#1\color{black}}}

\begin{document}
\title[RC-closed sublattices]
{2-distributivity and lattices of sublattices closed under taking relative complements}

\author[G.\ Cz\'edli]{G\'abor Cz\'edli}
\email{czedli@math.u-szeged.hu \qquad Address$:$~University of Szeged, Hungary}
\urladdr{http://www.math.u-szeged.hu/~czedli/}
\nothing{\address{ Bolyai Institute, University of Szeged, Hungary}}

\begin{abstract}
For a modular  lattice $L$ of finite length, we prove that the distributivity of $L$ is a sufficient condition while its 2-distributivity is a necessary condition that those  sublattices of $L$ that are closed under taking relative complements form a ranked lattice of finite length. 
\end{abstract}

\dedicatory{Dedicated to Professor B\'ela Cs\'ak\'any on his ninetieth birthday}

\thanks{This research was supported by the National Research, Development and Innovation Fund of Hungary under funding scheme K 134851.}



\date{\datum.\hfill{\qquad\red{{Check} the author's website (later MathSciNet or Zentralblatt) for updates and related papers}}}

\maketitle
\[\phantom{mmmmm}\red{
\boxed{\parbox{8.5cm}{\mgreen{Compared to the previous (January 6, 2022) arXiv version, now Theorem \ref{thmmaina}, Lemma \ref{lemma:rcgen}, and Corollary \ref{cor:nhhNskpS} are stronger since they do not assume the modularity of $L$.}}}}
\]

For elements $u,x,v$ of a lattice $L$, let 
$\rcp L(u,x,v):=\set{y\in L: x\wedge y=u\text{ and }x\vee y=v}
$. A sublattice $S$ of $L$ is \emph{closed with respect to taking relative complements} or, shortly saying, it is an \emph{RC-closed sublattice} if for all $u,x,v\in S$,  we have that  $\rcp L(u,x,v)\subseteq S$. The set consisting of the empty set and the RC-closed sublattices of $L$ will be denoted by $\RCSub L$. For any lattice $L$,  the poset $\RCSub L =\bigl(\RCSub L, \subseteq\bigr)$ is an algebraic lattice.
Following  Huhn \cite{huhn71,huhn72}, a lattice $L$ is \emph{$n$-distributive} if for all $x,y_0,\dots,y_n\in L$, $x\wedge\bigvee\set{y_i:0\leq i\leq n}=\bigvee_{j=0}^n\bigl( x\wedge\bigvee\set{y_i:0\leq i\leq n\text{ and }i\neq j}\bigr)$.   The \emph{length} $\len L$ of a lattice $L$ is the supremum of the lengths of its finite chains; for a finite chain $C$, $\len C=|C|-1$. A \emph{ranked lattice} is a lattice in which any two maximal chains of every principal ideal are of the same finite length.  Our aim is to prove the following four statements.

\begin{theorem}\label{thmmaina}
If $L$ is a lattice of finite length, then $\len{\RCSub L}=1+ \len L$.
\end{theorem}

\begin{theorem}\label{thmmainb}
If $L$ is modular lattice such that $\RCSub L$ is a ranked lattice of finite length, then $L$ is a $2$-distributive lattice of finite length.
\end{theorem}

\begin{theorem}\label{thmmainc}  
For every finite distributive lattice $L$,  $\RCSub L$ is a ranked lattice.
\end{theorem}

For a lattice $L$ and $X\subseteq L$, $\prcg L X$ stands for the least RC-closed sublattice of $L$ that includes $X$ as a subset; ``g'' in the acronym comes from ``generated''.

\begin{lemma}\label{lemma:rcgen}
If $L$ is a lattice of finite length,  $Y$ is an RC-closed sublattice of $L$, $X$ is a sublattice of\, $Y$, and 
$\len X=\len Y$, then $Y=\prcg L X$.
\end{lemma}

For the particular case where  $L$ is distributive,
Lemma \ref{lemma:rcgen} could be  extracted from Section 10 of Gr\"atzer \cite{gg1st}. Letting $Y:=L$, the lemma trivially implies the following.

\begin{corollary}\label{cor:nhhNskpS}
If $X$ is a sublattice of a lattice $L$ of finite length such that $\len X=\len L$, then $L$ is RC-generated by $X$.
\end{corollary}

\begin{proof}[Proof of Lemma \ref{lemma:rcgen}] 
We give a proof by contradiction. With $S:=\prcg L X$, suppose that $S\neq Y$. Since  $X\subseteq S\subset Y$,  $\len Y=\len X\leq \len S\leq \len Y$ gives that $\len S=\len Y$.  This allows us to pick a maximal chain $K$ in $S$ such that $\len K=\len Y$.
We have that  $0_K=0_S=0_Y$, $1_K=1_S=1_Y$, and both $K$ and $S$  are \emph{cover-preserving sublattices} of $Y$, that is,  for any $x,y\in S$, $x\prec_S y\iff x\prec_Y y$, and similarly for $K$. For $a\in L$, the principal ideal and the principal filter generated by $a$ are denoted by $\lideal L a:=\set{x\in L: x\leq a}$ and $
\lfilter L a:=\set{x\in L: x\geq a}$, respectively. 
We write $\ideal a$ and $\filter a$ if $L$ is understood.
If $u\leq v$ in $Y$, then the \emph{length} of the interval $[u,v]$, understood in $Y$, will be denoted by $\llen Y([u,v])$.
For, say, $u\leq v\in K$, the notation 
 $\llen K([u,v])$ is analogously defined. For $x\in Y$,  we let
\[\lcov x:=\bigvee (K\cap \lideal Y x) \quad\text{ and }\quad
\ucov x:=\bigwedge (K\cap \lfilter Y x).
\]
Since $K$ is finite and so it is a complete sublattice, $\lcov x,\ucov x\in K$ for every $x\in Y$. Choose an element $b\in Y\setminus S$ such that $d:=\llen K([\lcov b, \ucov b])=\llen Y([\lcov b, \ucov b])$ is minimal. Since $K\ni \lcov b\leq b \leq \ucov b\in K$
but $b\notin K$, $d\geq 2$. Hence, we can pick an element $q\in K$ such that $\lcov b<q<\ucov b$. Then $z:=b\wedge q\in Y$. 
If  $z$ was outside $S$, then $K\ni\lcov b\leq z\leq q\in K$
would give that $\lcov b\leq \lcov z\leq \ucov z\leq q<\ucov b$ and 
$\llen K([\lcov z, \ucov z]) <  \llen K([\lcov b, \ucov b])=d$ would be a contradiction. Thus, $b\wedge q=z\in S$. Dually, $b\vee q\in S$ and, of course, $q\in K\subseteq S$. Thus, $b\in \rcp L(b\wedge q,q,b\vee q)\in S$, contradicting that $b\in Y\setminus S$. 
\end{proof}

In the following three proofs, we can assume that  $n:=\len L$ is at least 2. As always in Lattice Theory, ``$\subset$'' is the conjunction of ``$\subseteq$'' and ``$\neq$''.

\begin{proof}[Proof of Theorem \ref{thmmaina}]
Let $0=c_0\prec c_1\dots\prec c_n=1$ be a chain in $L$. Let $X_{-1}:=\emptyset$ and, for $i\in\set{0,\dots,n}$, $X_i:=\ideal{c_i}$. Since all these $X_i$ belong to $\RCSub L$ and 
$X_{-1}\subset X_0\subset X_1\subset \dots\subset X_n=L$, 
 $\len {\RCSub L}\geq 1+n=1+\len L$. 
Next, let $Y_{-1}\subset Y_0\subset\dots \subset Y_k=L$ be an arbitrary chain in $\RCSub L$. Clearly,  $\len{Y_{i-1}}\leq \len{Y_{i}}$ for 
all $i\in\set{0,1,\dots, k}$. We claim that 
$\len{Y_{i-1}} < \len{Y_{i}}$ for all meaningful $i$.
Suppose the contrary. Then $Y_{i-1}\subset Y_i$ and $\len{Y_{i-1}} = \len{Y_{i}}$ for some $i$. With reference to Lemma~\ref{lemma:rcgen} at ``$=^\ast$'', $Y_{i-1}=\prcg L {Y_{i-1}} =^\ast Y_i$, which is a contradiction proving that 
$\len{Y_{i-1}} < \len{Y_{i}}$ for all meaningful $i$.
Therefore, since $-1\leq \len\emptyset\leq \len {Y_{-1}}$, we obtain that 
$k\leq 1+n$. Thus,  $\len {\RCSub L}\leq 1+n=1+\len L$, completing the proof.
\end{proof}

\begin{proof}[Proof of Theorem \ref{thmmainb}]
Following von Neumann \cite{neumann} and going also after Herrmann \cite{herrmannarith}, a system $F=(a_i, c_{i,j}: i,j\in\set{1,\dots,n}, i\neq j)$ of elements of a modular lattice $L$ is a \emph{non-trivial normalized (von Neumann) $n$-frame} or, briefly, a \emph{von Neumann $n$-frame} if, with the notation $0_F:=\bigwedge_{i=1}^n a_i$ and  $1_F:=\bigvee_{i=1}^n a_i$,  we have that $0_F\neq 1_F$, 
$a_j\wedge\bigvee_{t\neq j}a_t=0_F=a_i\wedge c_{i,j}$, $c_{i,j}=c_{j,i}$, $a_i\vee c_{i,j}=a_i\vee a_j$, and $c_{i,k}=(a_i\vee a_k)\wedge(c_{i,j}\vee c_{j,k})$ for all $\set{i,j,k}\subseteq\set {1,\dots,n}$  with $|\set{i,j,k}|=3$.
Here $2\leq n\in\Nplu:=\set{1,2,3,\dots}$. 
We know from  Huhn \cite[Proposition 1.2]{huhnsmqsts} that 
\begin{equation}
\parbox{8.8cm}{for $n\in\Nplu$, a modular lattice is $n$-distributive if and only if it does not contain a von Neumann $(n+1)$-frame.}
\label{pbx:huhnndfr}
\end{equation}
Subsection 1.4 "Reduction of frames" together with Subsection 1.7 of Herrmann and Huhn \cite{herrmannhuhnNSub} prove that, for $n\in\Nplu$, 
\begin{equation}\label{pbx:hrhnrdFr}
\parbox{10.0cm}{if $F=(a_i, c_{i,j}: i,j\in\set{1,\dots,n}, i\neq j)$ is a von Neumann $n$-frame in a modular lattice $L$ and $a'_1\in L$ such that $0_F<a'_1 < a_1$, then $a_1'$ belongs to a von Neumann $n$-frame $F=(a'_i, c'_{i,j}: i,j\in\set{1,\dots,n}, i\neq j)$ such that $0_F=0_{F'}<a'_i < a_i$ for $i\in\set{1,\dots,n}$.} 
\end{equation}

We say that a projective space is \emph{irreducible} (AKA nondegenerate) if each of its lines contains at least three points. It is known (and easy to see) that each point of an irreducible projective plane lies on at least three lines. By Huhn \cite[Thm.\ 1.1]{huhn2n},
\begin{equation}\label{pbx:huhn}
\parbox{9.5cm}{
 for $n\in\Nplu$,
a modular algebraic lattice is $n$-distributive if and only if none of its sublattices is isomorphic to the subspace lattice of an irreducible $n$-dimensional projective geometry.}
\end{equation}

Assume that $L$ is a modular lattice such that $\RCSub L$ is a ranked lattice of finite length. 
Whenever $\dots <  c_i <  c_{i+1}<\dots$ is a (finite or infinite) chain in  $L$, then $\dots < \prcg L{\ideal c_i} < \prcg L{\ideal{c_{i+1}}}<\dots$ is a chain in $\RCSub L$. Thus, $L$ is  of finite length. 
For the sake of contradiction, suppose that  $L$ is not 2-distributive.  By \eqref{pbx:huhnndfr}, $L$ contains a von Neumann 3-frame $F=(a_i,c_{i,j}:i\neq j,\,\,i,j\in\set{1,\dots,3})$. Applying \eqref{pbx:hrhnrdFr}, repeatedly if necessary, we can assume that $0_F\prec a_i$, for $i\in\set{1,2,3}$.  The definition of a $3$-frame implies that $\set{a_1,a_2,a_3}$ is an independent set of atoms in the filter $\filter {0_F}$. Hence, $1_F=a_1\vee a_2\vee a_3$ is of height 3 in $\filter {0_F}$, that is, $\llen L([0_F,1_F])=3$. By \eqref{pbx:huhnndfr},  the interval $I:=[0_F,1_F]$ is not $2$-distributive. Using \eqref{pbx:huhn}, we obtain that the subspace lattice $S$ of a projective plane $G$ is a sublattice of $I$. Since $\len I=3=\len S$, $S$ is a cover-preserving sublattice of $I$ and $L$,  $0_S=0_F=0_I$,  and $1_s=1_I$.  
Let $\At S$ and $\Coat S$ denote the set of atoms and that of coatoms of $S$, respectively. The are disjoint sets, $S=\set{0_S,1_S}\cup\At S\cup\Coat S$, $\At S=\set{a\in I: 0_I\prec a,\,\, a\in S}$, and dually. A trivial geometric argument shows that for $\forall a\in \At S$ and $\forall b\in\Coat S$,
$1_S\overset{\ast}=\bigvee\bigl(\At S\setminus\set {a}\bigr)$ and $0_S\overset{\ast\ast}=\bigwedge\bigl(\Coat S\setminus\set{b}\bigr)$.
Define $Z_{-1}:=\emptyset$, $Z_0:=\set{0_S}$, $Z_1:=\set{0_S, 1_S}$, and $Z_2:=\prcg L S$.  We claim that
\begin{align}
Z_{-1},Z_0,Z_1,Z_2=I\in\RCSub L\text{ and }Z_{-1}\prec Z_0\prec Z_1\prec Z_2\text{ in }\RCSub L. \label{eq:spHkwKq}
\end{align}
By Lemma~\ref{lemma:rcgen}, $I=\prcg I S$. 
 Since $\prcg I S\subseteq \prcg L S$, we obtain that $I\subseteq \prcg L S=Z_2$. Since every interval is RC-closed, 
$S\subseteq I\in\RCSub L$ yields that  $Z_2=\prcg L S\subseteq \prcg L I=I$. Thus, $Z_2=I$. Trivially, $Z_{-1}\prec Z_0\prec Z_1$ and $Z_1\subset Z_2$.  To verify that $Z_1\prec Z_2=I$,
assume that $Z_1\subset X\subseteq I$ for some $X\in\RCSub L$. 
Pick an element $u\in X\setminus Z_1$. Then $0_F=0_S < u < 1_S$.  Since $\len S=3$, either $u$ is of height 1, or it is of height 2. First, assume that $u$ is of height 2, that is, $u\prec 1_S$ in $I$ (and in $L$). If we had that $|\At S\setminus\ideal u|\leq 1$, then  ``$\overset{\ast}=$'' would give that $1_S=\bigvee \bigl(\At S\cap\ideal u\bigr)\leq u$, contradicting that $u<1_S$. Hence,  $|\At S\setminus\ideal u|\geq 2$, and we can pick two distinct elements, 
${v_1}$ and ${v_2}$, from  $\At S\setminus\ideal u$. For $s\in\set{1,2}$, using that ${v_s}\not\leq u$, $u,{v_s}\in [0_S,1_S]$, $0_S\prec {v_s}$, and $u\prec 1_S$ , we obtain that ${v_s}\in\rcp L(0_S, u, 1_S)$. Hence, ${v_s}\in X$.  
By modularity (in fact, by semimodularity), ${v_1}\prec {v_1}\vee {v_2}\in X$. Since $\len I=3$, $0_S=0_I \prec {v_1}\prec {v_1}\vee {v_2}< 1_S=1_I$. This chain, being in $X$, shows that $3\leq \len X$. On the other hand, $\len X\leq \len I=3$. 
Using Lemma \ref{lemma:rcgen} and that $X\in \RCSub L$, we have that $X=\prcg L X=I=Z_2$. Thus, 
$Z_1\prec Z_2$, completing the proof of \eqref{eq:spHkwKq}.

By \eqref{eq:spHkwKq}, $Z_{-1}\prec\dots\prec Z_2$ extends to a maximal chain $\vec Z: Z_{-1}\prec\dots\prec Z_k$ of $\RCSub L$. Like in the proof of Theorem~\ref{thmmaina}, Lemma~\ref{lemma:rcgen} applies and we obtain that $-1\leq \len{Z_{-1}} < \len{Z_{0}} < \dots < \len{Z_{k}}$. 
In fact, we know even that $\len{Z_{1}} +1 =2 < 3=\len{Z_{2}}$.
Thus, the maximal chain $\vec Z$ consists of at most $1+\len L$ elements, whence $\len{\vec Z}\leq \len L$. But $\RCSub L$ is a ranked lattice of finite length,  whereby $\len{\RCSub L}=\len{\vec Z}\leq \len L$,
contradicting  Theorem~\ref{thmmaina}. 
\end{proof}

\begin{proof}[Proof of Theorem \ref{thmmainc}]
Let $L$ be a finite distributive lattice. Since it is a ranked lattice of finite length,  it suffices to show that for any $U,V\in\RCSub L$,
\begin{equation}
U\prec V\text{ in }\RCSub L\,\,\iff\,\, \bigl(U\subset V\text{ and }\len V=\len U+1 \bigr).
\label{eq:msnClksh}
\end{equation}
To prove the $\Rightarrow$ direction, assume that $U\prec V$.
Clearly, $U\subset V$ and $\len U\leq \len V$. 
If we had that $\len U=\len V$, then Lemma~\ref{lemma:rcgen}  would give that $U=\prcg L U=V$, a contradiction. Hence, $\len U < \len V$. We are going to show by way of contradiction that $\len V=\len U+1$. Suppose the contrary; then $k:=\len U\leq \len V-2$. Take a maximal chain $C_0$ in $U$.
Since $\len {C_0}=k$ and $\len V\geq k+2$, we can extend $C_0$ to
a chain $C$ of $V$  such that $\len C=k+1$. 
Let $W:=\prcg L C$. It follows from Lemma \ref{lemma:rcgen} that 
$U=\prcg L {C_0}$. Hence, $U=\prcg L{C_0}\subseteq \prcg L C=W$.
Since $\len W\geq \len C>\len U$, $W\neq U$. Hence, $U\subset W$. 
The inclusion $C\subseteq V$ gives that $W=\prcg L C\subseteq  
\prcg L V=V$. Combining $U\subset W\subseteq V$ and $U\prec V$, we obtain that $W=V$.

Next, we  write $C$ in the form $C=\set{c_0<c_1<\dots<c_{k+1}}$. (Note that ``$<$'' here need not mean covering in $L$.) By Birkhoff \cite{birkhoff}, we can fix a finite Boolean lattice $D$ such that $L$ is a sublattice of $D$. We define the elements $b_i\in D$ for $i\in\set{1,\dots, k+1}$ by the rule $b_i\in \rcp D(c_0,c_{i-1},c_i)$. Since $D$ is a Boolean lattice, $b_i$ exists and it is uniquely determined. 
We claim that, for $i=2,3,\dots, k+1$,
\begin{equation}
c_0\notin\set{b_1,\dots,b_i},\text{ } c_{i-1}=b_1\vee\dots\vee b_{i-1},\text{ and }(b_1\vee\dots \vee b_{i-1})\wedge b_i=c_0.
\label{eq:kfbKszmD}
\end{equation} 
We show this by induction on $i$. Clearly, $b_1=c_1\neq c_0$.
Since $b_2\in\rcp D(c_0,c_1,c_2)$, we have that $b_2\neq c_0$ since otherwise $c_2=c_1\vee b_2=c_1$ would be a contradiction.
Also, $b_2\in\rcp D(c_0,c_1,c_2)$ gives that $b_1\wedge b_2=c_1\wedge b_2=c_0$. Hence,  
\eqref{eq:kfbKszmD} holds for $i=2$. Assume that $2\leq i<k+1$ and \eqref{eq:kfbKszmD} holds for this $i$.
As before,  $b_{i+1}\neq c_0$ since otherwise $b_{i+1}\in\rcp D(c_0,c_i,c_{i+1})$ would lead to $c_{i+1}=b_{i+1}\vee c_i=c_i$, which is a contradiction. Using the definition of $b_i$ and the
induction hypothesis, we have that
$c_i =c_{i-1}\vee b_i=  b_1\vee \dots \vee b_{i-1}\vee b_i$, that is, the second equality of \eqref{eq:kfbKszmD} holds for $i+1$. Using this equality and the definition of $b_{i+1}$, we have that 
$ (b_1\vee \dots\vee  b_i)\wedge b_{i+1}=c_i\wedge b_{i+1}=c_0$. This completes the induction step and proves that \eqref{eq:kfbKszmD} holds for $i=2,3,\dots,k+1$.

By, say, Gr\"atzer \cite[Theorem 360]{ggglt} and  \eqref{eq:kfbKszmD}, 
$\set{b_1,b_2,\dots,b_{k+1}}$ is a $(k+1)$-element independent set in the filter $\lfilter D {c_0}$. Thus, this set generates a $2^{k+1}$-element Boolean sublattice $E$. Using that any element in an interval of a distributive lattice has at most one relative complement with respect to the interval in question and $E$ as a Boolean lattice is closed with respect to taking relative complements, we obtain that $E$ is RC-closed. It is clear from 
\eqref{eq:kfbKszmD} and $b_{k+1}\in\rcp D(c_0,c_{k},c_{k+1})$ that 
$C\subseteq E$. Hence, $\prcg D C\subseteq \prcg D E=E$. 
Now let $F$ be a maximal chain in $W$. Since 
$W=\prcg L C\subseteq \prcg D C\subseteq E$,  we have that $F$ is a chain in $E$. But $\len E=k+1$, implying that $\len F\leq k+1$.
So $\len W=\len F\leq k+1$. On the other hand, $k+1= \len C\leq \len W$. Thus, $\len W=k+1$, which is a contradiction since $W=V$ and $\len V=k+2$. This proves the $\Rightarrow$ direction of \eqref{eq:msnClksh}.

To prove the $\Leftarrow$ direction, assume that $U\subset V$ and $\len V=\len U+1$. Assume also that $H\in \RCSub L$ such that $U\subseteq H\subseteq V$. Clearly, $\len U\leq \len H\leq \len V$, whence $\len H$ is $\len U$ or $\len V$.  
If $\len H=\len U$, then Lemma \ref{lemma:rcgen} gives that
$U=\prcg L U=H$. Similarly, if $\len H=\len V$, then the same lemma gives that $H=\prcg L H=V$. Therefore, $U\prec V$, completing the proof.
\end{proof}

Some facts are easy to observe; see Cz\'edli~\cite{CzGjnm} for details. For a lattice $L$, $\RCSub L$ is Boolean $\iff$ $\RCSub L$ is semimodular $\iff$ $\RCSub L$ is lower semimodular. Let $B_n$ stand for the $2^n$-element Boolean lattice. Our tools lead to a (bit complicated)  formula for $|\RCSub {B_{n}}|$, which allows us to compute (in $0.05$ seconds with computer algebra) that, say, $|\RCSub {B_{57}}|$ is the 59-digit number
\[
46\,669\,606\,977\,325\,325\,544\,440\,157\,525\,187\,321\,911\,338\,002\,625\,473\,546\,541\,556
\]
where $B_{57}$ is the Boolean lattice of length 57. It took 24 seconds to obtain that $|\RCSub {B_{999}}|\approx 0. 566 \,759\, 343 \, 075 \, 881 \, 648 \cdot 10^{1930}$.


\begin{thebibliography}{99}
\bibitem{CzGjnm} 
 Cz\'edli, G.: A property of lattices of sublattices closed under taking relative complements and its connection to 2-distributivity. In preparation
\end{thebibliography}
\end{document}